\documentclass[11pt, a4paper]{amsart}
\usepackage{amscd}
\usepackage{amsmath}
\usepackage{amsfonts}

\newcommand{\de}{\partial}

\newcommand{\Ric}{\mathrm{Ric}}

\newcommand{\db}{\overline{\partial}}

\newcommand{\C}{\mathbb{C}}
\newcommand{\mn}{\sqrt{-1}}

\newcommand{\ti}[1]{\tilde{#1}}
\newcommand{\ve}{\varepsilon}

\newcommand{\diam}{\mathrm{diam}}

\begin{document}
\newcounter{remark}
\newcounter{theor}
\setcounter{remark}{0}
\setcounter{theor}{1}
\newtheorem{claim}{Claim}
\newtheorem{thm}{Theorem}[section]
\newtheorem{prop}{Proposition}[section]
\newtheorem{lemma}{Lemma}[section]
\newtheorem{defn}{Definition}[theor]
\newtheorem{cor}{Corollary}[section]
\newenvironment{remark}[1][Remark]{\addtocounter{remark}{1} \begin{trivlist}
\item[\hskip
\labelsep {\bfseries #1  \thesection.\theremark}]}{\end{trivlist}}
\setlength{\arraycolsep}{2pt}

\title{K\"ahler-Ricci flow on stable Fano manifolds}
\author{Valentino Tosatti}
 \address{Department of Mathematics \\ Harvard University \\ Cambridge, MA 02138}

  \email{tosatti@math.harvard.edu}
\begin{abstract} 
We study the K\"ahler-Ricci flow on Fano manifolds. We show that if the curvature is bounded along the flow and if the manifold is K-polystable and asymptotically Chow semistable, then the flow converges exponentially fast to a K\"ahler-Einstein metric.
\end{abstract}
\maketitle
\section{Introduction}
The Ricci flow was introduced by Hamilton in \cite{Ha1} and is a way to deform a Riemannian metric by a nonlinear parabolic evolution equation. On a compact complex manifold $M$ if the initial metric is K\"ahler then the flow will preserve this property, and is thus called the K\"ahler-Ricci flow. It was first studied by Cao in \cite{Co}, where he showed that a solution of the normalized flow exists for all positive time, assuming the first Chern class of the manifold is definite and the initial metric lies in the correct class. 
If the first Chern class is zero or negative the behavior of the flow is perfectly understood \cite{Ya}, \cite{aubin}, \cite{Co}, so we will restrict our attention to compact K\"ahler manifolds with positive first Chern class, that is Fano manifolds.
The normalized K\"ahler-Ricci flow equation then takes the form
\begin{equation}\label{metricke}
\frac{\de}{\de t}\omega_t=-\Ric(\omega_t)+\omega_t,
\end{equation}
where $\omega_t$ is a family of K\"ahler metrics cohomologous to
$c_1(M)$. The fixed points of the flow are K\"ahler-Einstein metrics with positive scalar curvature, that is metrics $\omega_{KE}$ that satisfy
\begin{equation}\label{keke}
\Ric(\omega_{KE})=\omega_{KE}.
\end{equation}
If such a metric exists then it is unique up to the action of $\mathrm{Aut}^0(M)$, the connected component of the identity of the biholomorphism group of $M$ \cite{BM}. In general there are obstructions to the existence of K\"ahler-Einstein metrics, and these fall into two categories: obstructions arising from $\mathrm{Aut}^0(M)$, such as Matsushima's Theorem \cite{Ma} or the Futaki invariant \cite{Fu}, and obstructions arising from stability. A long-standing conjecture of Yau \cite{Ya2} says that a K\"ahler-Einstein metric should exist precisely when the manifold, polarized by the anticanonical bundle $K_M^{-1}$, is stable in a suitable algebro-geometric sense. The precise notion of stability involved is called {\em K-polystability}, has been introduced by Tian \cite{Ti3} and refined by Donaldson \cite{D3}. The Yau-Tian-Donaldson's conjecture then states that the existence of a K\"ahler-Einstein metric on a Fano manifold $M$ is equivalent to K-polystability of $(M,K_M^{-1})$. There has been much progress on the subject, see for example \cite{Si}, \cite{Ti1},  \cite{Ti2}, \cite{TY}, \cite{Na}, \cite{D1}, \cite{Ti3}, \cite{WZ}, \cite{ZZ} but the conjecture is still open in general. A natural approach to this conjecture is to show that the K\"ahler-Ricci flow \eqref{metricke} converges to a K\"ahler-Einstein metric. Since we know that the flow \eqref{metricke} exists for all time the issue is to show that stability implies convergence of the flow at infinity. Despite some recent powerful estimates of Perelman \cite{Pe}, \cite{ST}, this seems to be out of reach at present. On the other hand some progress has been done under the assumption that the curvature remains bounded along the flow. In \cite{PS2}, \cite{PSSW2}, \cite{Sz} it is shown that if this holds and if the manifold is stable in some different analytic ways, then the flow converges to a K\"ahler-Einstein metric. For more recent work on the K\"ahler-Ricci flow and K\"ahler-Einstein metrics, the reader can consult for example \cite{CT}, \cite{TZ}, \cite{Ch}, \cite{PSS}, \cite{Ru}, \cite{Zhu}, \cite{SoT}. This is not an exhaustive list in any sense, and we refer to the above works for more references.
Our main result is the following (the relevant definitions are in section 3):
\begin{thm}\label{main}
Let $M$ be a compact complex manifold with $c_1(M)>0$ (that is a Fano manifold). Assume that along the K\"ahler-Ricci flow \eqref{metricke} the sectional curvatures remain bounded
\begin{equation}\label{curvature}
|\mathrm{Rm}_t|\leq C.
\end{equation}
Assume moreover that
$(M,K_M^{-1})$ is K-polystable and asymptotically Chow semistable. Then the
flow converges exponentially fast to a
K\"ahler-Einstein metric on $M$.
\end{thm}

Let us comment on the hypotheses of our theorem. The assumption that the curvature remains bounded is certainly very strong, as it basically only allows smooth manifolds as geometric limits at infinity. But there are some cases when we know that it is satisfied: if $M$ is a Fano surface then by \cite{TY}, \cite{Ti2}, \cite{Ko1}, \cite{WZ} it admits a K\"ahler-Ricci soliton. If it is also K-polystable then the soliton must be trivial, that is K\"ahler-Einstein. Then a result of Perelman (see \cite{TZ}) implies that the flow \eqref{metricke} converges smoothly to a K\"ahler-Einstein metric, so in particular \eqref{curvature} holds. On the other hand, if $M$ admits a metric with nonnegative bisectional curvature $\omega_0$ then the flow starting from it will also have this property \cite{Ba}, \cite{Mo}, and Perelman's result that the scalar curvature remains bounded (see \cite{ST}) implies that \eqref{curvature} holds. Of course it would be desirable to have a proof that on Fano surfaces \eqref{curvature} holds, without using the existence of K\"ahler-Einstein metrics. 

Let us also remark that according to \cite{RZZ} if $n\geq 3$ then \eqref{curvature} follows from the weaker bound
\begin{equation}\label{ln}
\int_M |\mathrm{Rm}_t|^n\omega_t^n\leq C.
\end{equation}
Notice that when $n=2$ the bound \eqref{ln} always holds thanks to Perelman's scalar curvature bound \cite{ST}. 

Another point to notice is that in our theorem we need to assume asymptotic Chow semistability as well as K-polystability. This is not really a problem because Chow semistability is an algebraic GIT notion and so it is in line with Yau's conjecture. Moreover it is very likely that asymptotic Chow semistability is implied by K-polystability (see \cite{RT}), especially when $M$ has no nonzero holomorphic vector fields because of \cite{D1}. 

Let us say a few words about the strategy of proof of the main theorem. The boundedness of curvature together with some results of Perelman imply that we can take convergent subsequences of the flow as time goes to infinity, and get as limits K\"ahler-Ricci solitons on Fano manifolds diffeomorphic to $M$. The complex structures that we get are in general different from the original one, but the original one is a sufficiently small deformation of them. If any of these solitons happens to be K\"ahler-Einstein then the assumption of K-polystability together with a theorem of Sz\'ekelyhidi \cite{Sz} imply that the original manifold also admits a K\"ahler-Einstein metric. Otherwise if all the solitons are nontrivial we show that the Mabuchi energy must be unbounded below along the flow, and in fact it must decrease at least linearly. Then we show that the Mabuchi energy is suitably approximated by a sequence of functionals $\ti{\mathcal{L}}_m$, originally defined by Donaldson in \cite{D2}. This approximation result relies on showing that the so-called Tian-Yau-Zelditch-Catlin expansion holds uniformly along the flow, a result that uses \eqref{curvature}. Finally the assumption of asymptotic Chow semistability is shown to be equivalent to the fact that the functionals $\ti{\mathcal{L}}_m$ are bounded below for $m$ large, and this contradicts the decay of the Mabuchi energy.

The organization of the paper is as follows: in section 2 we set up the notation and we state some results from K\"ahler geometry and Ricci flow that we will use. In section 3 we link the algebraic condition of Chow semistability to some energy functionals. In section 4 we prove our main Theorem \ref{main}.\\

\noindent
{\bf Acknowledgments.} I thank my advisor Prof. Shing-Tung Yau for his constant support and help. I am also grateful to G\'abor Sz\'ekelyhidi for many stimulating discussions and to Ben Weinkove for useful comments. Part of this work was carried out when I was visiting the Morningside Center of Mathematics in Beijing, which I thank for the gracious hospitality. I was also partially supported by a Harvard Merit Fellowship. These results will form part of my forthcoming PhD thesis at Harvard University.

\section{Notation and Basic Results}
In this section we set up the notation used throughout the paper, and we recall some results about the K\"ahler-Ricci flow that will be used extensively.\\

In this paper $(M,J)$ denotes a compact complex manifold of complex dimension $n$ and with positive first Chern class $c_1(M)>0$. We will often drop the reference to the complex structure $J$. For any K\"ahler metric $\omega\in c_1(M)$ we can write it locally as
$$\omega=\frac{\mn}{2\pi} g_{i\overline{j}}dz^i\wedge d\overline{z}^j,$$
where $g=(g_{i\overline{j}})$ is the Hermitian metric determined by $\omega$ and $J$. Here we are using the Einstein summation convention. The volume form associated to $\omega$ is $\frac{\omega^n}{n!}$, and we will denote by $V$ the volume of $M$
$$V=\int_M\frac{\omega^n}{n!}.$$
We will write
$\Delta_\omega$ for the Laplacian of $g$, which acts on a function $F$ as
$$\Delta_\omega F=g^{i\overline{j}}\frac{\de^2 F}{\de z^i\de \overline{z}^j}.$$ 
The Ricci curvature of $\omega$ is the tensor locally defined by
$$R_{i\overline{j}}=-\frac{\de}{\de z^i}\frac{\de}{\de \overline{z}^j}
\log\det(g),$$
and we associate to it the Ricci form
$$\Ric(\omega)=\frac{\mn}{2\pi} R_{i\overline{j}}dz^i\wedge d\overline{z}^j.$$
It is a closed real $(1,1)$-form that represents the cohomology class $c_1(M)\in H^2(M,\mathbb{Z})$. The scalar curvature of $\omega$ is denoted by
$$R=g^{i\overline{j}}R_{i\overline{j}}.$$
To the metric $\omega$ we can associate
its Ricci potential $f_\omega$, which is the real function defined by $$\Ric(\omega)=\omega+\mn\de\db f_\omega,$$
and $$\int_M (e^{f_\omega}-1)\frac{\omega^n}{n!}=0.$$
The space of K\"ahler potentials of the metric $\omega$ is the set of all smooth real functions $\varphi$ such that $\omega_\varphi=\omega+\mn\de\db\varphi$ is a K\"ahler metric.
Then we can define a real-valued functional $F^0_\omega$ on the space of K\"ahler potentials by the formula
$$F^0_\omega(\varphi)=-\frac{1}{V}\int_0^1\int_M \frac{\de \varphi_t}{\de t} \frac{\omega_{\varphi_t}^n}{n!},$$
where $\varphi_t$ is any smooth path of K\"ahler potentials with $\varphi_0=0$ and $\varphi_1=\varphi$ (for example one can take $\varphi_t=t\varphi$).
It can be written also as
\begin{equation}\label{f0}
F^0_\omega(\varphi)=J_\omega(\varphi)-\frac{1}{V}\int_M\varphi\frac{\omega^n}{n!},
\end{equation}
where the functional $J_\omega$ is defined by
$$J_\omega(\varphi)=\frac{1}{V}\int_0^1\int_M \frac{\de \varphi_t}{\de t}\left(\frac{\omega^n}{n!}-\frac{\omega_{\varphi_t}^n}{n!}\right),$$
and integration by parts shows that $J_\omega(\varphi)\geq 0$. Moreover 
$F^0_\omega$ satisfies the following cocycle condition
\begin{equation}\label{cocycle}
F^0_\omega(\varphi)=F^0_\omega(\psi)+F^0_{\omega_\psi}(\varphi-\psi),
\end{equation}
for all K\"ahler potentials $\varphi, \psi$.

We now consider the space $H^0(M,K_M^{-m})$ of holomorphic sections 
of the $m$th plurianticanonical bundle, where $m\geq 1$. This is a vector space whose
dimension $N_m$ can be computed from the Riemann-Roch formula, when $m$ is large
\begin{equation}\label{rr}
N_m=\int_M \mathrm{ch}(K_M^{-m})\wedge\mathrm{Todd}(M)\approx Vm^n+\frac{nV}{2}m^{n-1}+O(m^{n-2}).
\end{equation}
Let us fix $h$ a Hermitian metric along the fibers of $K_M^{-1}$ with curvature equal to $\omega$. This induces metrics $h^m$ on the tensor powers $K_M^{-m}$.
For each given positive integer $m$ we also fix $\{S_i\}$ a basis of $H^0(M,K_M^{-m})$ which is orthonormal with respect to the $L^2$ inner product defined by $h^m$, $\omega^n$:
$$\int_M \langle S,T\rangle_{h^m}\frac{\omega^n}{n!}.$$
Then we can define the ``density of states'' function 
\begin{equation}\label{rr2}
\rho_m(\omega)=\sum_{i=1}^{N_m}|S_i|^2_{h^m}.
\end{equation}
It does not depend on the choice of orthonormal basis $\{S_i\}$ or on the choice of $h$, and so it is canonically attached to $\omega$ and $J$. Its name stems from the property that
\begin{equation}\label{rr3}
\int_M \rho_m(\omega)\frac{\omega^n}{n!}=N_m.
\end{equation}
The Tian-Yau-Zelditch-Catlin expansion is the following 
\begin{thm}[Zelditch \cite{Ze}, Catlin \cite{Ct}]\label{zc} When $m$ is large we have an expansion
\begin{equation}\label{rr4}
\rho_m(\omega)\approx m^n+a_1(\omega)m^{n-1}+a_2(\omega)m^{n-2}+\dots,
\end{equation}
where $a_i(\omega)$ are smooth functions defined locally by $\omega$, and the expansion is valid in any $C^k(\omega)$ norm. More precisely this means that 
given any $k,N\geq 1$ there is a constant $C$ that depends only on $k,N,\omega$ such that
$$\left\|\rho_m(\omega)-m^n-\sum_{i=1}^N a_i(\omega)m^{n-i} \right\|_{C^k(\omega)}\leq C m^{n-N-1},$$
for all $m\geq 1$.
\end{thm}
Moreover Z. Lu \cite{L} has computed that $a_1(\omega)=\frac{R}{2}$.
The expansion \eqref{rr4} integrates term by term to the Riemann-Roch expansion \eqref{rr}.

Now we turn to the K\"ahler-Ricci flow \eqref{metricke}. The initial metric of the flow is denoted by $\omega$ and it will be considered as a reference metric.
Since $\omega_t$ and $\omega$ are cohomologous, we can write
$\omega_t=\omega+\mn\de\db\varphi_t$, where the normalization of the
potentials $\varphi_t$ will be specified presently. Then
\eqref{metricke} is equivalent to the following parabolic
Monge-Amp\`ere equation
\begin{equation}\label{ke}
(\omega+\mn\de\db\varphi_t)^n= e^{f_\omega-\varphi_t+\dot\varphi_t}\omega^n,
\end{equation}
where $\dot\varphi_t$ denotes $\de\varphi_t/\de t$, and $\varphi_0$ equals the constant 
\begin{equation}\label{norm}
\frac{1}{V}\int_M f_\omega \frac{\omega^n}{n!}+\int_0^\infty 
e^{-t}\left(\int_M |\nabla \dot{\varphi}_t|^2_{t} \frac{\omega_t^n}{n!}
\right)dt,
\end{equation}
and it is easy to see that this is well defined.
The choice of $\varphi_0$ determines the normalization of
$\varphi_t$ for $t>0$, and is necessary to get convergence of the flow since it is the only choice that ensures that $\dot{\varphi}_t$ is bounded at infinity (see \cite{PSS}, \cite{CT} and \eqref{perelman} below).
We now state Perelman's estimates for the K\"ahler-Ricci flow on Fano manifolds, and we refer to the exposition \cite{ST} and to \cite{Pe} for proofs.
\begin{thm}[Perelman] With the above choice of $\varphi_0$,
there is a constant $C$ that depends only on $\omega$ such that for all $t\geq 0$ we have
\begin{equation}\label{perelman}
|\dot{\varphi}_t|+\diam(M,\omega_t)+|R_t|\leq C,
\end{equation}
where $R_t$ denotes the scalar curvature of $\omega_t$.
Moreover given any $r_0>0$ there is a constant $\kappa>0$ that depends only on $r_0$ and $\omega$ such that for all $t\geq 0$, all $p\in M$ and all $0<r<r_0$ we have
\begin{equation}\label{perelman2}
\int_{B_t(p,r)}\frac{\omega_t^n}{n!}\geq \kappa r^{2n},
\end{equation}
where $B_t(p,r)$ is the geodesic ball in the metric $\omega_t$ centered at $p$ of radius $r$.
\end{thm}

\section{Chow Semistability}
In this section we link Chow semistability to a certain functional on the space of K\"ahler potentials. The results in this section follow from work of Donaldson \cite{D2} and S. Zhang \cite{ZhS}.\\

Let us first recall the definition of K-polystability that we will use, which is a special case of Tian's definition \cite{Ti3}. A test configuration for $(M,K_M^{-1})$ is a $\mathbb{C}^*$-equivariant flat family $(\mathcal{M},\mathcal{L})$ over $\mathbb{C}$ such that the generic fiber is isomorphic to $(M, K_M^{-m})$ for some $m>0$. The central fiber $(M_0, L_0)$ inherits a $\mathbb{C}^*$-action, and we assume that $M_0$ is a smooth manifold and so $L_0=K_{M_0}^{-m}$. K-polystability then means that for any test configuration we require the Futaki invariant of the vector field generating the action on $M_0$ to be nonnegative, and equal to zero only when $M_0$ is biholomorphic to $M$. There is a more general definition of Donaldson \cite{D3} that allows singular central fibers but in our situation these are excluded by assumption \eqref{curvature}. In fact in Theorem \ref{main} instead of K-polystability we only need to assume that there is no test configuration with smooth central fiber $M_0$ which has zero Futaki invariant and is not isomorphic to $M$.

Now we recall the definition of asymptotic Chow semistability (we refer the reader to \cite{Wa} for details). 
For each $m$ sufficiently large the line bundle $K_M^{-m}$ is very ample, and so choosing a basis $\{S_i\}$ of holomorphic sections in $H^0(M,K_M^{-m})$ gives an embedding of $M$ inside $\mathbb{P}^{N_m-1}=\mathbb{P}H^0(M,K_M^{-m})^*.$ Associated to this embedding there is a point $\mathrm{Chow}_m(M)$ in the Chow variety of cycles in $\mathbb{P}^{N_m-1}$ of dimension $n$ and degree $d=Vm^n n!$. If we let $\mathbb{G}$ be the Grassmannian of $N_m-n-2$-planes in $\mathbb{P}^{N_m-1}$ and if we call $W=H^0(\mathbb{G},\mathcal{O}(d))$, then 
the Chow variety sits inside the projective space $\mathbb{P}(W^*)$. This projective space has a linearized action of $SL(N_m,\mathbb{C})$, inherited from the natural action on $\mathbb{G}$, which changes $\mathrm{Chow}_m(M)$ by changing the basis $\{S_i\}$. Then Chow semistability of $(M,K_M^{-m})$ means that the $SL(N_m,\mathbb{C})$ orbit of the point $\mathrm{Chow}_m(M)$ is GIT semistable (one can likewise define Chow polystability and stability). Asymptotic Chow semistability then means Chow semistability of $(M,K_M^{-m})$ for all $m$ sufficiently large. By the Kempf-Ness theorem \cite{KN} Chow semistability of $(M,K_M^{-m})$ is equivalent to the fact that the
function 
\begin{equation}\label{below}
\tau\mapsto\log\frac{\|\tau\cdot
\mathrm{Chow}_m(M)\|^2}{\|\mathrm{Chow}_m(M)\|^2}
\end{equation}
 is bounded below on
$SL(N_m,\C)$. Here $\|\cdot\|$ is any norm on the vector space $W$ which is invariant under $SU(N_m)$. 

We now fix $h$ a metric on $K_M^{-1}$ with curvature equal to $\omega$, and for each $m$ we also fix $\{S_i\}$ a basis of $H^0(M,K_M^{-m})$ which is orthonormal with respect to the $L^2$ inner product defined by $h^m$, $\omega^n$. Given a matrix $\tau\in GL(N_m,\C)$ we
 define the corresponding ``algebraic K\"ahler potential'' by
$$\varphi_{\tau}=\frac{1}{m}\log\frac{\sum_{i}| \sum_j\tau_{ij} S_j|^2_{h^m}}{\sum_i |S_i|^2_{h^m}}.$$
This has the following interpretation. We use the sections $\{S_i\}$ to embed $M$
inside $\mathbb{P}H^0(M,K_M^{-m})=\mathbb{P}^{N_m-1}$. This carries a natural K\"ahler form $\omega_{FS}$, the Fubini-Study form associated to the $L^2$ inner product of $h^m, \omega^n$. If we let $\tau$ act
on $\mathbb{P}^{N_m-1}$ via the natural action, then on $M$ we have
\begin{equation}\label{pullback2}
\tau^*\omega_{FS}=\omega_{FS}+m\mn\de\db\varphi_\tau,
\end{equation}
so $\varphi_\tau$ is a K\"ahler potential for $\frac{\omega_{FS}}{m}$.
On the other hand, we also have that
\begin{equation}\label{pullback}
\omega=\frac{\omega_{FS}}{m}-\frac{1}{m}\mn\de\db\log\rho_m(\omega),
\end{equation}
and so the function
$$\psi_{\tau}=\frac{1}{m}\log\sum_{i}\bigl|\sum_j\tau_{ij}S_j\bigr|^2_{h^m}=\varphi_\tau+
\frac{1}{m}\log\rho_m(\omega)$$
is a K\"ahler potential for $\omega$. 
Now if we go back to \eqref{below} and we choose the norm $\|\cdot\|$ suitably (see \cite{PS1}) then a theorem of Zhang
\cite{ZhS} (see also \cite{Pau,PS1}) gives that
\begin{equation}\label{chowf}
F^0_{\frac{\omega_{FS}}{m}}(\varphi_{\tau})=-\frac{1}{Vm(n+1)}
\log\frac{\|\tau \cdot
\mathrm{Chow}_m(M)\|^2}{\|\mathrm{Chow}_m(M)\|^2},
\end{equation}
for all $\tau\in SL(N_m,\C)$.
We now introduce a slight variant of $F^0_\omega$, following Donaldson \cite{D2}. Given a K\"ahler potential $\varphi$ we let $h_\varphi=h e^{-\varphi}$, which is a metric on $K_M^{-1}$ with curvature equal to $\omega_\varphi$.
The $L^2$ inner product on $H^0(M,K_M^{-m})$ defined by $h_\varphi^m$, $\omega_\varphi^n$, can be represented as a positive definite Hermitian matrix, with respect to the fixed basis $\{S_i\}$. Explicitly, this means that we
set
\begin{equation}\label{hma}
H_{i\overline{j},\varphi}=\int_M \langle S_i, S_j\rangle_{h^m_\varphi}\frac{\omega_\varphi^n}{n!}.
\end{equation}
We then let
$$c_\varphi=\log |\det H_{i\overline{j},\varphi}|.$$
Notice that changing the basis $\{S_i\}$ does not affect $c_\varphi$, which depends only on $\varphi$ and the choice of $h$. Also $c_\varphi$ changes smoothly if $\varphi$ does. We then define the functional
$$\ti{\mathcal{L}}_m(\varphi)=\frac{c_\varphi}{N_m}-mF^0_\omega(\varphi).$$
If $\varphi_t$ is a smooth path of K\"ahler potentials then the variation of $\ti{\mathcal{L}}_m$ can be computed as follows: since $c_t=c_{\varphi_t}$ is independent of the choice of $S_i$ we can pick them so that for a fixed time $t$ we have $H_{i\overline{j},t}=\lambda_i^2 \delta_{ij}$ where the numbers $\lambda_i$ are real and nonzero. Then the holomorphic sections $T_i=\frac{S_i}{\lambda_i}$ are orthonormal with respect to $h_{t}^m$, $\omega_{t}^n$, and we have
\begin{equation}
\begin{split}
\frac{\de}{\de t}c_t&=H^{i\overline{j}}_t \frac{\de}{\de t}H_{i\overline{j},t}=\sum_i\frac{1}{\lambda_i^2}\int_M |S_i|^2_{h^m}
\frac{\de}{\de t}\left(e^{-m\varphi_t}\frac{\omega_t^n}{n!}\right)\\
&=\sum_i\frac{1}{\lambda_i^2}\int_M |S_i|^2_{h_t^m}(-m\dot{\varphi}_t+\Delta_t
\dot{\varphi}_t)\frac{\omega_t^n}{n!}\\
&=\sum_i \int_M |T_i|^2_{h_t^m}(-m\dot{\varphi}_t+\Delta_t
\dot{\varphi}_t)\frac{\omega_t^n}{n!}\\
&=\int_M \rho_m(\omega_t)(-m\dot{\varphi}_t+\Delta_t
\dot{\varphi}_t)\frac{\omega_t^n}{n!}\\
&=\int_M \dot{\varphi_t}(\Delta_t\rho_m(\omega_t)-m\rho_m(\omega_t))\frac{\omega_t^n}{n!}.
\end{split}
\end{equation}
So we get
\begin{equation}\label{varl}
\frac{\de}{\de t}\ti{\mathcal{L}}_m(\varphi_t)=\frac{1}{N_m}\int_M \dot{\varphi}_t 
\left(\Delta_t\rho_m(\omega_t)-m\rho_m(\omega_t)+\frac{mN_m}{V}\right)
\frac{\omega_t^n}{n!}.
\end{equation}
We then have the following
\begin{prop}\label{lowerbo}
The pair $(M,K_M^{-m})$ is Chow semistable if and only if there exists a constant $C$, that might depend on $m$, such that 
\begin{equation}\label{lowerb}
\ti{\mathcal{L}}_m(\varphi)\geq -C,
\end{equation}
for all K\"ahler potentials $\varphi$.
\end{prop}
In the rest of the paper we will only use one implication, but we include the proof of both for completeness.
\begin{proof}
First we prove that Chow semistability implies the lower boundedness of 
$\ti{\mathcal{L}}_m$.
For each potential $\varphi$ set $h_\varphi= he^{-\varphi}$ and $\omega_\varphi=\omega+\mn\de\db\varphi$. Choose $\{S_i(\varphi)\}$ a basis of $H^0(M,K_M^{-m})$ which is orthonormal with respect to the $L^2$ inner product defined by $h_\varphi^m$, $\omega_\varphi^n$. Then we can write
$$S_i(\varphi)=\sum_j \tau_{ij}S_j,$$
for some matrix $\tau=(\tau_{ij})\in GL(N_m,\C)$ that depends on $\varphi$. From the definition we get
$$c_\varphi=-2\log|\det\tau|,$$
and so
\begin{equation}\label{useful1}
\ti{\mathcal{L}}_m(\varphi)=-\frac{2}{N_m}\log|\det\tau|-mF^0_\omega(\varphi).
\end{equation} 
We now observe that from the cocycle formula \eqref{cocycle} we have
\begin{equation}\label{useful2}
\begin{split}
F^0_\omega(\varphi)-F^0_\omega(\psi_{\tau})&=-F^0_{\omega_\varphi}(\psi_{\tau}
-\varphi)\\
&=-J_{\omega_\varphi}(\psi_{\tau}-\varphi)+
\frac{1}{V}\int_M(\psi_{\tau}-\varphi)\frac{\omega_\varphi^n}{n!}\\
&\leq\frac{1}{V}\int_M(\psi_{\tau}-\varphi)\frac{\omega_\varphi^n}{n!},
\end{split}
\end{equation}
where we have also used the fact that $J_{\omega_\varphi}\geq 0$.
On the other hand we have
$$\int_M e^{m(\psi_{\tau}-\varphi)}\frac{\omega_\varphi^n}{n!}=
\sum_i\int_M |S_i(\varphi)|^2_{h^m}e^{-m\varphi}\frac{\omega_\varphi^n}{n!}=N_m,$$
and so by Jensen's inequality
$$\frac{m}{V}\int_M(\psi_{\tau}-\varphi)\frac{\omega_\varphi^n}{n!}\leq \log(N_m/V).$$
Together with \eqref{useful2} this gives
\begin{equation}\label{useful3}
F^0_\omega(\varphi)-F^0_\omega(\psi_{\tau})\leq \frac{\log(N_m/V)}{m}.
\end{equation}
Then the cocycle formula \eqref{cocycle} gives
\begin{equation}\label{useful8}
F^0_{\frac{\omega_{FS}}{m}}(\varphi_{\tau})=F^0_{\omega}(\psi_{\tau})+
F^0_{\frac{\omega_{FS}}{m}}\left(-\frac{1}{m}\log\rho_m(\omega)\right),
\end{equation}
and this together with \eqref{useful3} gives
\begin{equation}\label{useful4}
-mF^0_\omega(\varphi)\geq -m F^0_\omega(\psi_{\tau})-C\geq 
-mF^0_{\frac{\omega_{FS}}{m}}(\varphi_{\tau})-C.
\end{equation}
This and \eqref{useful1} give
\begin{equation}\label{useful5}
\ti{\mathcal{L}}_m(\varphi)\geq -\frac{2}{N_m}\log|\det\tau|-mF^0_{\frac{\omega_{FS}}{m}}(\varphi_{\tau})-C.
\end{equation}
We now set 
$$\ti{\tau}= (\det\tau)^{-\frac{1}{N_m}}\tau,$$
so now $\ti{\tau}\in SL(N_m,\C)$ and we notice that
$$\varphi_{\ti{\tau}}=\varphi_{\tau}-\frac{2}{mN_m}\log|\det\tau|,$$
and so
\begin{equation}\label{useful6}
-\frac{2}{N_m}\log|\det\tau|-mF^0_{\frac{\omega_{FS}}{m}}(\varphi_{\tau})=
-mF^0_{\frac{\omega_{FS}}{m}}(\varphi_{\ti{\tau}})\geq -C,
\end{equation}
by Chow semistability and \eqref{chowf}. Combining \eqref{useful5} with \eqref{useful6} finally gives
\begin{equation}\label{bddd}
\ti{\mathcal{L}}_{m}(\varphi)\geq -C.
\end{equation}
To show the other implication we assume that \eqref{bddd} holds and we let $\tau$ be any matrix in $SL(N_m,\C)$.
By \eqref{chowf} it is enough to prove that the function
$$-mF^0_{\frac{\omega_{FS}}{m}}(\varphi_{\tau})$$
has a uniform lower bound independent of $\tau$. By \eqref{useful8} we have
$$-mF^0_{\frac{\omega_{FS}}{m}}(\varphi_{\tau})\geq -m F^0_\omega(\psi_{\tau})-C,$$
and we have
$$\ti{\mathcal{L}}_m(\psi_\tau)=\frac{c_{\psi_\tau}}{N_m}-mF^0_\omega(\psi_\tau),$$
so we are reduced to showing that $c_{\psi_\tau}$ is bounded above independent of $\tau$. Notice that in \eqref{hma} if we use the basis $\ti{S}_i=\sum_j \tau_{ij}S_j$ instead of $S_i$ we get a different matrix 
$\ti{H}_{i\overline{j},\psi_\tau}$ but its log determinant is the same.
Using the definitions we have that
$$\omega_{\psi_\tau}=\frac{\tau^*\omega_{FS}}{m},$$
and
$$\ti{H}_{i\overline{j},\psi_\tau}=\int_M\frac{\langle \ti{S}_i,\ti{S}_j\rangle_{h^m}}
{\sum_{k}|\sum_l\tau_{kl}S_l|^2_{h^m}}\cdot\frac{(\tau^*\omega_{FS})^n}{m^n n!},$$
so using the arithmetic-geometric mean inequality we get
$$\frac{c_{\psi_\tau}}{N_m}\leq \log\left(\frac{1}{N_m}\int_M
\frac{\sum_i |\ti{S}_i|^2_{h^m}}{\sum_i |\ti{S}_i|^2_{h^m}}\cdot
\frac{(\tau^*\omega_{FS})^n}{m^n n!}\right),$$
and since the integral above is just the volume of $M$ in the metric $\frac{\tau^*\omega_{FS}}{m}$, this is bounded independent of $\tau$.
\end{proof}

\section{The Main Theorem}
In this section we prove Theorem \ref{main} by relating the behavior of the functionals $\ti{\mathcal{L}}_m$ to the Mabuchi energy.\\

\begin{proof}[Proof of Theorem \ref{main}]
First of all we use Perelman's estimate \eqref{perelman2}: this together
 with \eqref{curvature}, Perelman's diameter bound \eqref{perelman} and Theorem 4.7 of \cite{CGT} gives a uniform lower bound for the injectivity radius of $(M,\omega_t)$ independent of $t$. Then Hamilton's compactness theorem \cite{Ha2} gives that for any sequence $t_i\to\infty$
we can find a subsequence (still denoted $t_i$), a K\"ahler structure $(\omega_\infty, J_\infty)$ on the differentiable manifold $M$ and diffeomorphisms $F_i:M\to M$
such that $\omega_i=F_i^*\omega_{t_i}\to \omega_\infty$ and
$J_i={F_i^{-1}}_*\circ J\circ {F_i}_*\to J_\infty$ smoothly.
We will denote by $\de_i$ (resp. $\de_\infty$) the $\de$-operators of $J_i$ (resp. $J_\infty$).
An argument of \v{S}e\v{s}um-Tian (see \cite{ST} or \cite{PSSW2} p. 662) shows that
$(\omega_\infty, J_\infty)$ is a K\"ahler-Ricci soliton, and so it satisfies
\begin{equation}\label{soliton}
\Ric(\omega_\infty)=\omega_\infty+\mn\de_\infty\db_\infty\psi,
\end{equation}
for a smooth function $\psi$ whose gradient is a $J_\infty$-holomorphic vector field. Such a function $\psi$ is only defined up to addition of a constant, but we can choose it by requiring that
$$\int_M (e^\psi-1)\frac{\omega_\infty^n}{n!}=0.$$
Since along the flow we have that
$$\Ric(\omega_t)=\omega_t+\mn\de\db f_t,$$
where $f_t$ is the Ricci potential of $\omega_t$,
it follows that the functions $F_i^* f_{t_i}$ will converge smoothly to $\psi$.
In fact $F_i^*\de\db f_{t_i}=\de_i\db_i F_i^*f_{t_i}$ converges smoothly to
$\de_\infty\db_\infty \psi$, and the statement follows because of the normalizations we chose.
Notice that $f_t$ is equal to $\dot\varphi_t$ up to a constant.

We let $\mathcal{M}_\omega(\varphi_t)$ be the Mabuchi energy, normalized so that $\mathcal{M}_\omega(\varphi_0)=0.$
Recall that the variation of the Mabuchi energy is
$$\frac{\de}{\de t}\mathcal{M}_\omega(\varphi_t)=-\frac{1}{V}\int_M \dot{\varphi}_t(R_t-n)\frac{\omega_t^n}{n!},$$
while the variation of $\ti{\mathcal{L}}_m$ was computed in \eqref{varl} to be
$$\frac{\de}{\de t}\ti{\mathcal{L}}_m(\varphi_t)=\frac{1}{N_m}\int_M \dot{\varphi}_t 
\left(\Delta_t\rho_m(\omega_t)-m\rho_m(\omega_t)+\frac{mN_m}{V}\right)
\frac{\omega_t^n}{n!}.$$
For a fixed metric $\omega$ the Tian-Yau-Zelditch-Catlin expansion \eqref{rr4} says that as $m\to\infty$ we have
$$\rho_m(\omega)\approx m^n+\frac{R}{2}m^{n-1}+O(m^{n-2}),$$
Recalling that by Riemann-Roch \eqref{rr} we also have
$$N_m\approx Vm^n+\frac{nV}{2}m^{n-1}+O(m^{n-2}),$$
we get that for a fixed metric $\omega$
$$\Delta\rho_m(\omega)-m\rho_m(\omega)+\frac{mN_m}{V}\approx 
\frac{m^{n}}{2}(n-R)+O(m^{n-1}).$$
We claim that this still holds uniformly along the flow.
\begin{prop}\label{zeld}
Given any $k, m_0$ and $\ve>0$ there exist an $m\geq m_0$ and a $t_0>0$ such that
for all $t\geq t_0$ we have
$$\frac{1}{m^{n-1}}\left\|\rho_m(\omega_t)-m^n-\frac{R_t}{2}m^{n-1}\right\|_{C^k(\omega_t)}
\leq\ve.$$
\end{prop}
The proof of this proposition is postponed.
As above applying this with $k=2$ and $m_0, \ve$ to be specified later, we get
\begin{equation}\label{uno}
\frac{1}{m^n}|\Delta_t\rho_m(\omega_t)|\leq 
\frac{\ve}{m}+\frac{|\Delta_t R_t|}{2m},
\end{equation}
while Riemann-Roch implies that
\begin{equation}\label{due}
\left|\frac{N_m}{Vm^{n-1}}-\left(m+\frac{n}{2}\right)\right|\leq\frac{C_0}{m}.
\end{equation}
Proposition \ref{zeld} also gives that
\begin{equation}\label{tre}
\left|\frac{\rho_m(\omega_t)}{m^{n-1}}-m-\frac{R_t}{2}\right|\leq\ve.
\end{equation}
Putting together \eqref{uno}, \eqref{due} and \eqref{tre} gives
$$\left|\frac{1}{m^n}\left(\Delta_t\rho_m(\omega_t)-m\rho_m(\omega_t)+\frac{mN_m}{V}
\right)-\frac{1}{2}(n-R_t)\right|\leq 2\ve+
\frac{C_0}{m}+\frac{|\Delta_t R_t|}{2m}.$$
From the boundedness of curvature and Shi's estimates \cite{Sh}, it follows that for all $t\geq 0$ we have
$$|\Delta_t R_t|\leq C_1,$$
for a uniform constant $C_1$.

Using now the fact that $|\dot{\varphi}_t|\leq C_2$ and Riemann-Roch, we get
$$\left| \frac{\de}{\de t}\left(\ti{\mathcal{L}}_m-\frac{\mathcal{M}_\omega}{2}\right)(\varphi_t)\right|\leq
2(2\ve+C_3/m)\frac{1}{V}\int_M|\dot{\varphi}_t|\frac{\omega_t^n}{n!}\leq C_2(2\ve+C_3/m).$$
In particular given any $\ve_1>0$ we can fix $\ve$ and $m_0$ so that
$$C_2(2\ve+C_3/m_0)\leq \ve_1,$$
and moreover $(M,K_M^{-m})$ is Chow semistable for all $m\geq m_0$.
Then Proposition \ref{zeld} with the above arguments gives an $m\geq m_0$ and a $t_0$ such that for all $t\geq t_0$ 
$$\left| \frac{\de}{\de t}\left(\ti{\mathcal{L}}_m-\frac{\mathcal{M}_\omega}{2}\right)(\varphi_t)\right|\leq \ve_1.$$
Integrating this, we get that for all $t\geq t_0$ we have
\begin{equation}\label{bound}
\ti{\mathcal{L}}_m(\varphi_t)\leq\frac{\mathcal{M}_\omega(\varphi_t)}{2}+\ve_1 t+C.
\end{equation}

We now claim that either $M$ already admits a K\"ahler-Einstein metric, or there is a constant $\gamma>0$ such that 
\begin{equation}\label{decay}
\frac{\de}{\de t}\mathcal{M}_\omega(\varphi_t)\leq -\gamma,
\end{equation}
for all $t$ sufficiently large.
In fact, if the above estimate fails, then we can find a sequence of times
$t_i\to\infty$ such that 
$$\frac{\de}{\de t}\mathcal{M}_\omega(\varphi_{t_i})>-\frac{1}{i}.$$
Since 
$$\frac{\de}{\de t}\mathcal{M}_\omega(\varphi_{t_i})=-\frac{1}{V}\int_M|\nabla\dot\varphi_{t_i}|^2_{t_i}\frac{\omega_{t_i}^n}{n!},$$
we get
\begin{equation}\label{small}
\frac{1}{V}\int_M|\nabla\dot\varphi_{t_i}|^2_{t_i}\frac{\omega_{t_i}^n}{n!}<\frac{1}{i}.
\end{equation}
By passing to a subsequence, we may assume that there are diffeomorphisms  $F_i:M\to M$
such that $\omega_i=F_i^*\omega_{t_i}\to \omega_\infty$ a K\"ahler-Ricci soliton as above. Then we have
\begin{equation}\label{small2}
\int_M|\nabla\dot\varphi_{t_i}|^2_{t_i}\frac{\omega_{t_i}^n}{n!}=\int_M |\nabla (F_i^*f_{t_i})|^2_{t_i} \frac{F_i^*\omega_{t_i}^n}{n!}\to\int_M |\nabla \psi|^2_\infty \frac{\omega_\infty^n}{n!},
\end{equation}
so by \eqref{small} we see that the Ricci potential $\psi$ of $\omega_\infty$ must be constant, and so $\omega_\infty$ is a K\"ahler-Einstein metric on 
$J_\infty$ a complex structure on $M$ of which $J$ is a small deformation. Notice that since $\omega_\infty$ is
K\"ahler-Einstein its cohomology class is $c_1(M,J_\infty)$. Moreover we have that the Chern classes $c_1(M,J_i)\to c_1(M,J_\infty)$ and since they are integral classes, we must have $c_1(M,J_i)=c_1(M,J_\infty)$ for all $i$ large.
So we can assume that the canonical bundles $K_{M,i}$ are all isomorphic to $K_{M,\infty}$ as complex line bundles, but with different holomorphic structures. So $(M,J_i, K_{M,i}^{-m})$ is a small deformation of $(M,J_\infty, K_{M,\infty}^{-m})$. Then the fact that $M$ is K-polystable together with Theorem 2 of \cite{Sz} shows that $M$ admits a K\"ahler-Einstein metric. The claim is proved.

Now we assume that $M$ does not admit a K\"ahler-Einstein metric, so that \eqref{decay} holds. We now pick $\ve_1<\gamma/2$, and consequently get an $m$ such that \eqref{bound} holds. But we are also assuming that $m$ is large enough, so that $(M,K_M^{-m})$ is Chow semistable and so by Proposition \ref{lowerbo} we have that \eqref{lowerb} holds.
We can integrate \eqref{decay}, which holds for all $t$ large, and
get
\begin{equation}\label{bdd}
\mathcal{M}_\omega(\varphi_t)\leq -\gamma t+C.
\end{equation}
and this together with \eqref{bound}, \eqref{lowerb} gives
$$-C\leq \ti{\mathcal{L}}_{m}(\varphi_t)\leq \frac{\mathcal{M}_\omega(\varphi_t)}{2}+\ve_1 t+C\leq -(\gamma/2-\ve_1)t+C,$$
for all $t$ large, which is absurd. Hence $M$ must admit a K\"ahler-Einstein metric. Once we know this, results of Perelman-Tian-Zhu \cite{TZ} and Phong-Song-Sturm-Weinkove \cite{PSSW1} imply that the flow converges exponentially fast. In fact, we can avoid the analysis of Perelman-Tian-Zhu in our case: the Theorem in \cite{Sz} that we used constructs a K\"ahler-Einstein metric $g_{KE,i}$ on $(M,J_i)$ for $i$ large as a small $C^\infty$ perturbation of a K\"ahler-Einstein metric $g_{KE,\infty}$ on $(M,J_\infty)$ (here we use the notation $g$ instead of $\omega$ to emphasize that we are considering the Riemannian metrics). In particular the Gromov-Hausdorff distance of $g_{KE,i}$ to $g_{KE,\infty}$ goes to zero as $i$ goes to infinity. But the metrics $(F_i^{-1})^*g_{KE,i}$ are then K\"ahler-Einstein on $(M,J)$ and by the Bando-Mabuchi uniqueness theorem \cite{BM} they must be all isometric to a fixed K\"ahler-Einstein metric $g_{KE}$ on $(M,J)$. Since their Gromov-Hausdorff distance to $g_{KE,\infty}$ is arbitrarily small, it follows that the Gromov-Hausdorff distance between $g_{KE}$ and $g_{KE,\infty}$ is zero, and so they are isometric. By Matsushima's theorem \cite{Ma} the space of holomorphic vector fields of $(M,J)$ is the complexification of the space of Killing vector fields of $g_{KE}$, but this is the same as the space of Killing vector fields of $g_{KE,\infty}$. It follows that $(M,J)$ and $(M,J_\infty)$ have the same dimension of holomorphic vector fields. By the argument in the proof of Theorem 5 in \cite{Sz} this implies that $J$ must be biholomorphic to $J_\infty$. So we have shown that there is a sequence of times $t_i$ and diffeomorphisms $F_i$ such that the metrics $F_i^*\omega_{t_i}$ converge smoothly to a K\"ahler-Einstein metric on $(M,J)$.
Then by a theorem of Bando-Mabuchi \cite{BM} the Mabuchi energy $\mathcal{M}_\omega$ has a lower bound, and the arguments in section 2 of \cite{PS2} show that $$\frac{\de}{\de t}\mathcal{M}_\omega(\varphi_t)\to 0,$$
as $t\to\infty$. This together with the above arguments imply that given any sequence $t_i\to\infty$ we can find a subsequence, still denoted $t_i$, and diffeomorphisms $F_i$ such that $F_i^*\omega_{t_i}$ converges smoothly to a K\"ahler-Einstein metric on $(M,J)$. A contradiction argument then implies that the flow converges modulo diffeomorphisms: there exists $\omega_\infty$ a K\"ahler-Einstein metric on $(M,J)$ and diffeomorphisms $F_t:M\to M$ such that $F_t^*\omega_t$ conveges smoothly to $\omega_\infty$.
Then \cite{PSSW1} shows that the original flow $\omega_t$ converges to a K\"ahler-Einstein metric exponentially fast.
\end{proof}
\begin{proof}[Proof of Proposition \ref{zeld}]
If the conclusion is not true, then there are a $k, m_0$ and $\ve_0>0$ such that for all $m\geq m_0$ and $i\geq 1$ there is a $t_i\geq i$ such that
$$\frac{1}{m^{n-1}}\left\|\rho_{m}(\omega_{t_i})-{m}^n-
\frac{R_{t_i}}{2}{m}^{n-1}\right\|_{C^k(\omega_{t_i})}\geq \ve_0.$$
Up to a subsequence, we may assume that $t_i\to\infty$ and that there are diffeomorphisms $F_i:M\to M$
such that $\omega_i=F_i^*\omega_{t_i}$ and
$J_i={F_i^{-1}}_*\circ J\circ {F_i}_*$ converge smoothly to some limit 
$\omega_\infty$ and $J_\infty$ respectively.
We remark that while the complex structures $J_i$ are all biholomorphic to each other, they might not be biholomorphic to $J_\infty$. Notice also that
the $C^k$ norms involved are invariant under diffeomorphisms, and that
$F_i^*R_{t_i}=R_i$, $F_i^*\rho_m(\omega_{t_i})=\rho_m(\omega_i)$. 
The Tian-Yau-Zelditch-Catlin expansion applied to $\omega_\infty$ 
gives that there exists a uniform constant $C$ such that for all $m$ we have
$$\frac{1}{m^{n-1}}\left\|\rho_{m}(\omega_{\infty})-{m}^n-
\frac{R_{\infty}}{2}{m}^{n-1}\right\|_{C^k(\omega_\infty)}\leq\frac{C}{m}.$$
Moreover since $\omega_i$ converges smoothly to $\omega_\infty$, it follows that the $C^k$ norms they define are uniformly equivalent, so we will also have
$$\frac{1}{m^{n-1}}\left\|\rho_{m}(\omega_{\infty})-{m}^n-
\frac{R_{\infty}}{2}{m}^{n-1}\right\|_{C^k(\omega_i)}\leq\frac{C}{m}.$$
Then we get
\begin{equation*}
\begin{split}
\ve_0&\leq \frac{1}{m^{n-1}}\left\|\rho_{m}(\omega_{i})-{m}^n-
\frac{R_{i}}{2}{m}^{n-1}\right\|_{C^k(\omega_i)}\\
&\leq
\frac{1}{2}\|R_i-R_\infty\|_{C^k(\omega_i)}+\frac{1}{m^{n-1}}\|\rho_m(\omega_i)-\rho_m(\omega_\infty)\|_{C^k(\omega_i)}
+\frac{C}{m}.
\end{split}
\end{equation*}
We now fix $m\geq m_0$ such that $C/m\leq \ve_0/4$. Since $\omega_i$ converges smoothly to $\omega_\infty$, when $i$ is sufficiently large we will have
$$\frac{1}{2}\|R_i-R_\infty\|_{C^k(\omega_i)}\leq \frac{\ve_0}{4}.$$
We now claim that when $i$ is large we will also have
$$\frac{1}{m^{n-1}}\|\rho_m(\omega_i)-\rho_m(\omega_\infty)\|_{C^k(\omega_i)}\leq\frac{\ve_0}{4},$$
which will give a contradiction. In fact we will show that, for $m$ fixed as above, the function $\rho_m(\omega_i)$ converges smoothly to $\rho_m(\omega_\infty)$ as $i$ goes to infinity. This can be done in several ways, for example using the implicit function theorem, or the $L^2$ estimates for the $\bar{\de}$ operator. We choose the first way because it is easier, though the second way gives more precise estimates.
As remarked earlier we can assume that the canonical bundles $K_{M,i}$ are all isomorphic to $K_{M,\infty}$ as complex line bundles, but with different holomorphic structures. Fix $h_\infty$ a metric on $K_{M,\infty}^{-1}$ with curvature $\omega_\infty$, and perturb it to a family of metrics $h_i$ on $K_{M,i}^{-1}$ with curvature $\omega_i$ that converge smoothly to $h_\infty$. Given $S$ a holomorphic section of $K_{M,\infty}^{-m}$ we wish to perturb $S$ to a family $S_i$ of holomorphic sections of $K_{M,i}^{-m}$ that converge smoothly to $S$.
Once this is done, it is clear that $\rho_m(\omega_i)$ converges smoothly to $\rho_m(\omega_\infty)$, and we are done.
Since $m$ can be assumed to be large, by Riemann-Roch we can write the dimension of 
$H^0(M,K_{M,\infty}^{-m})$ as an integral over $M$ of Chern forms of $\omega_\infty$:
$$\dim H^0(M,K_{M,\infty}^{-m})=\int_M \mathrm{ch}(K_M^{-m})\wedge\mathrm{Todd}(M,\omega_\infty,J_\infty),$$
 and we can do the same for $\omega_i, J_i$. But since $\omega_i$ and $J_i$ converge smoothly to $\omega_\infty$ and $J_\infty$, it follows that the Riemann-Roch integrals are equal, and so the dimension of 
$H^0(M,K_{M,\infty}^{-m})$ is the same as $N_m$. This means that we have a sequence of elliptic operators, $\square_{\bar{\de}_i}$, acting on $\Gamma(M,K_M^{-m})$ (more precisely on Sobolev $W^{r,2}$ sections of this line bundle with $r\geq n+k+1$ to make them $C^k$) which converge smoothly to $\square_{\bar{\de}_\infty}$ (which acts on the same space) and such that the dimension of their kernel is the same as in the limit. Then Lemma 4.3 in \cite{Ko2}, which is a simple consequence of the implicit function theorem, ensures that any element in the kernel of $\square_{\bar{\de}_\infty}$ can be smoothly deformed to a sequence of elements in the kernel of $\square_{\bar{\de}_i}$. 
\end{proof}

\end{document}